\documentclass[11pt,twoside]{article}
\usepackage{latexsym}
\usepackage{amssymb,amsbsy,amsmath,amsfonts,amssymb,amscd}
\usepackage{epsfig, graphicx,color}
\usepackage{hyperref}
\newcommand{\commentout}[1]{}
\newcommand{\RR}{\mathbb{R}}

\newcommand {\e}  {\epsilon}

\newcommand {\lb} {\lambda}
\newcommand {\Chi} {{\bf \raise 2pt \hbox{$\chi$}} }
\newcommand{\ud}{\,\mathrm{d}}

\newcommand{\bd}[1]{\boldsymbol{#1}}
\newcommand {\f}   {\frac}
\newcommand {\p}   {\partial}
\newcommand{\barint}{\kern3pt \raise3.4pt\hbox{\vrule height.6pt
    width7pt} \kern-10pt \int}
\newcommand{\dis}{\displaystyle}

\newcommand{\beq}{\begin{equation}}
\newcommand{\eeq}{\end{equation}}
\newcommand{\bea} {\begin{array}{rl}}
\newcommand{\eea} {\end{array}}
\newcommand{\bepa}{\left\{ \begin{array}{l}}
\newcommand{\eepa} {\end{array}\right.}
\newtheorem{theorem}{Theorem}[section]
\newtheorem{lemma}[theorem]{Lemma}

\newcommand{\qed}{{ \hfill
                       {\unskip\kern 6pt\penalty 500 \raise -2pt\hbox{\vrule\vbox to 6pt{\hrule width 6pt
                       \vfill\hrule}\vrule} \par}   }}
\title{\Large \bf Derivation of the bacterial run-and-tumble kinetic equation from a model with biochemical pathway}

\author{
Beno\^ \i t Perthame\thanks{Sorbonne Universit\'es, UPMC Univ Paris 06, UMR 7598, Laboratoire Jacques-Louis Lions, F-75005, Paris, France}
\thanks{CNRS, UMR 7598, Laboratoire Jacques-Louis Lions, F-75005, Paris, France}
\thanks{INRIA-Paris-Rocquencourt, EPC MAMBA, Domaine de Voluceau, BP105, 78153 Le Chesnay Cedex, France}
\and
Min Tang\thanks{Institute of natural sciences and department of mathematics,
Shanghai Jiao Tong University, Shanghai, 200240, China.
This author is partially supported by NSF of Shanghai under grant 12ZR1445400, NSFC 11301336
and 91330203, and Shanghai Pujiang Program 13PJ140700.}
\and
Nicolas Vauchelet\footnotemark[1] \footnotemark[2] \footnotemark[3]
}

\date{\today}

\begin{document}
\maketitle
\pagestyle{plain}
\pagenumbering{arabic}

\begin{abstract}
Kinetic-transport equations are, by now, standard models to describe the dynamics of populations of bacteria moving by run-and-tumble. Experimental observations show  that bacteria increase their run duration when encountering an increasing gradient of chemotactic molecules. This led to a first class of models which heuristically include tumbling frequencies depending on the path-wise gradient of chemotactic signal.

More recently,  the biochemical pathways regulating the flagellar motors were uncovered. This knowledge gave rise to a second class of kinetic-transport equations, that takes into account an intra-cellular molecular content  and which relates the tumbling frequency to this information. It turns out that the tumbling frequency depends on the chemotactic signal, and not on its gradient.

For these two classes of models, macroscopic equations of Keller-Segel type, have  been derived using diffusion or hyperbolic rescaling. We complete this program by showing how the first class of equations can be derived from the second class with molecular content after appropriate rescaling. The main difficulty is to explain why the path-wise gradient of chemotactic signal can arise in this asymptotic process.

Randomness of receptor methylation events can be included, and our approach can be used to compute the tumbling frequency in presence of such a noise.

\end{abstract}

\bigskip

\\[3mm]

\section{Introduction}

Two classes of kinetic-transport equations have been proposed to describe, at the cell scale, the movement of bacteria by `run and tumble' in a given external effective  signal $M(x,t)$, usually related to the extra-cellular chemo-attractant concentration $S$ by a relation of the type $M = m_0+\ln(S)$.

The simplest class is for the probability $\bar p(\bd{x},\bd{v},t)$ to find a bacteria at location $\bd{x} \in \RR^d$ and with velocity $\bd{v}\in V$ (a smooth bounded subset of $\RR^d$, one can choose the unit ball to fix idea ). The evolution of this probability is given by a Boltzmann type equation
\beq\label{eq:kinetic}
{\p_t} \bar	p + \bd{v}\cdot\nabla_{\bd{x}} \bar p = {\mathcal T}[D_t M] ( \bar p), \qquad \bd{x} \in \RR^d, \; \bd{v} \in V, \; t \geq 0,
\eeq
with the path-wise gradient of $M$ defined as
\beq\label{eq:DM}
D_t M= {\p_t}M	+\bd{v}\cdot\nabla_{\bd{x}}M
\eeq
and $T$ a tumbling kernel which typically takes the form
\beq\label{eq:tumbling1}
 {\mathcal T}[D_t M] (\bar p) = \int_V\left[  T \big(D_t M(\bd{x},\bd{v}',t), \bd{v},\bd{v}' \big) \bar p(\bd{x},\bd{v}',t)
  - T \big(D_t M(\bd{x},\bd{v},t),\bd{v}',\bd{v}\big) \bar p(\bd{x},\bd{v},t) \right] \ud\bd{v}' .
\eeq
Such equations, with $ {\mathcal T}$ depending on $M$ or $D_t M$,  were used intensively to model bacterial chemotaxis, possibly  with $M$ connected to the cell density, as a result of chemoattractant release by bacteria.  They were first introduced in \cite{ODA} and the Keller-Segel  drift-diffusion system was subsequently derived \cite{OH} in the diffusion limit; surprisingly, with a kernel $T$ depending on $M$ and not on its gradient, and in opposition to the Keller-Segel system which solutions blow-up for large mass, it was proved that the solutions exist globally \cite{CMPS,HKS}. However, experiments show that bacteria as {\em E.coli} extend their runs when feeling an increasing concentration of  chemoattractant and  this led to study tumbling kernels $T$ that depend on $D_tM$, see \cite{ErbanOthmer04, DS}. The nonlinear theory is then more difficult  (see \cite{bcgp} and the references therein) and blow-up can occur in finite time \cite{BC}.  These models with $T$ depending on $D_t M$ are able to explain the experimental observation of traveling pulses of bacteria, which cannot be done when $T$ only depends on $M$ itself, see~\cite{SC1, SC2}. Also, departing from this kinetic-transport equation, it is possible to rescale it and study the diffusion and hyperbolic limit as in~\cite{ErbanOthmer04, DS, SC1, SC2, HKS, HP}. When $T$ undergoes  stiff dependency on $D_t M$,  the hyperbolic limit is singular and the analysis is particularly delicate~\cite{JV1}.
\\

More elaborated  kinetic models have been proposed recently that  incorporate intracellular chemo-sensory system. In the simplest description of the biochemical pathways, they use a single additional variable $m \geq 0$, which represents the intracellular methylation level. Then,  the kinetic-transport equation is written for the probability density function $p(\bd{x},\bd{v},m,t)$ of
bacteria at time $t$, position $\bd{x}\in\RR^d$, moving at velocity $\bd{v}\in V$
and methylation level $m >0 $
\beq\label{eq:kineticm} \begin{cases}
{\p_t}p+\bd{v}\cdot\nabla_{\bd{x}} p+\p_m [f(m,M)p] = {\mathcal Q}[m,M](p),
\\[5pt]
p(\bd{x},\bd{v},m=0,t)=0.
\end{cases}
\eeq
The intracellular adaptation dynamics is described by the reaction rate $f(\cdot)$ for which we assume $f(m=0,M)>0$, which allows us to pose the boundary condition at $m=0$.
The tumbling term ${\mathcal Q}[m,M](p)$ is
\begin{equation}\label{eq:Q}
{\mathcal Q}[m,M](p)=\int_V \left[
\lambda (m,M,\bd{v},\bd{v}')p(t,\bd{x},\bd{v}',m)  -
 \lambda (m,M,\bd{v}',\bd{v}) p(t,\bd{x},\bd{v},m) \right] \ud\bd{v}',
\end{equation}
where $\lambda(m,M,\bd{v},\bd{v}')$ denotes the methylation dependent tumbling frequency from
$\bd{v}'$ to $\bd{v}$, in other words the response of the cell depending on its environment and internal state. We borrow this formalism from \cite{JOT, STY} even though this type of models, involving more general signal transduction, can be traced back to \cite{DS, ErbanOthmer04, ErbanOthmer07, XO}. The authors in \cite{DS,ErbanOthmer04, ErbanOthmer07, STY,XO, X} developed the asymptotic theory which allows to recover, in the diffusion and in the hyperbolic limits, macroscopic equations where the variables are only $(\bd{x},t)$ as the Keller-Segel system, or $(\bd{x},m,t)$ for structured Keller-Segel models.
\\

In the program of establishing the relations between these pieces of the model hierarchy for bacterial population motion, a derivation is missing: how are related these two classes of kinetic models~\eqref{eq:kinetic}--\eqref{eq:tumbling1} and \eqref{eq:kineticm}--\eqref{eq:Q}?
\\

Our goal is to show how, assuming  fast adaptation and stiff response, the methylation level is at equilibrium with the external signal represented by $M$, and the equation~\eqref{eq:kinetic} can be derived from~\eqref{eq:kineticm}.  In particular we aim at computing the bulk tumbling kernel $T(D_t M, \bd{v}, \bd{v'})$ from the methylation dependent kernel $\lambda (m,M,\bd{v},\bd{v'})$, a statement we give in the next section. Two difficulties arise here: one is to infer the proper rescaling in the kinetic equations, the second is to carry-out the mathematical analysis for singular limits. Our approach allows us to also  include noise resulting from random receptor-methylation and demethylation events. The proof of the formula for $T$ is given in sections~\ref{sec:change} and \ref{sec:apb}; we show that a direct use of the variable $m$ is not enough to produce the formula and that a new variable is needed, which zooms on the  intra- and extra-cellular methylation equilibrium. We conclude by relating our notations to a more physically based description of the same model where the cell receptors activity is used in the model parameters, see section~\ref{sec:physics}.

To keep simplicity, we assume that the external signal function $M(\bd{x},t)$ is given and smooth. Therefore questions of existence and blow-up are not considered here.

\section{Fast adaptation, stiff response}

\paragraph{Assumptions.}
For our mathematical derivation, we introduce a small parameter $\e$ which acts both as a fast time scale for external signal transduction and as a stiffness parameter for the response in terms of  tumbling rate.
We assume moreover that the reaction rate $f$ only depends on the difference $m-M$, in accordance with the physical models that we recall in Section~\ref{sec:physics}.
Therefore, we rescale equation~\eqref{eq:kineticm}--\eqref{eq:Q} as

\beq\label{eq:kineticscale}
\begin{cases}
{\p_t}p_\epsilon+\bd{v}\cdot\nabla_{\bd{x}} p_\epsilon+\f{1}{\epsilon}\p_m\Big(f\big(m-M\big)p_\epsilon\Big)= {\mathcal Q_\e}[m,M](p_\e),
\\[5pt]
p_\e (\bd{x},\bd{v},m=0,t)=0,
\end{cases}
\eeq
with the tumbling kernel
\beq\label{eq:kineticscaleQ}
 {\mathcal Q_\e}[m,M](p_\e) = \int_V \left[
\Lambda \Big(\f{m-M}{\epsilon}, \bd{v},\bd{v}' \Big)p_\epsilon(\bd{x},\bd{v}',m,t) - \Lambda \Big(\f{m-M}{\epsilon}, \bd{v}',\bd{v}\Big)p_\epsilon(\bd{x},\bd{v},m,t) \right] \ud\bd{v}' .
\eeq
We complete this equation with an initial data  $p^{\rm ini} \geq 0$ which satisfies
\beq \label{as:pinitmoment}
\dis \iiint_{\RR^d\times V\times \RR} (1+m^2) p^{\rm ini}(\bd{x},\bd{v},m) \ud \bd{x} \ud \bd{v} \ud m<\infty ,
\eeq
\beq \label{as:barpinit}
\bar{p}^{\rm ini} :=\int_{\mathbb{R}}p^{\rm ini} \ud m  \in L^\infty( \RR^d\times V).
\eeq

Also, we are going to use several assumptions for the functions $M$, $f$ and $\Lambda$. We assume they are as smooth as necessary and
that for some constants $m_\pm$, $g_\pm$, $\lambda_\pm$,
\beq\label{eq:M}
0 < m_- \leq M(\bd{x},t) \leq m_+, \qquad M\in C_{\rm b}^1 \big( \RR^d \times [0, \infty) \big),
\eeq
\beq\label{eq:F}
f(y)=-yG(y),\qquad \mbox{with } G\in C_{\rm b}^1( \RR), \qquad   0<g_-\leq G(y) \leq g_+,
\eeq
\beq\label{eq:lambda}
\p_y \Lambda \Big(y, \bd{v},\bd{v}' \Big) <0, \qquad  0 < \lambda_- \leq   \Lambda \Big(y, \bd{v},\bd{v}' \Big)  \leq \lambda_+.
\eeq

Various scalings have been proposed for kinetic equation and the closest, but still different seems to be  the high field limit \cite{BeMaPo}. In \cite{Liao_Jie}, still other scalings or limits are studied.

\paragraph{The main result.}

With these assumptions, we are going to show that as $\epsilon$ vanishes, we recover the simpler model~\eqref{eq:kinetic}--\eqref{eq:tumbling1} as a limit of~\eqref{eq:kineticscale}.
\begin{theorem} [Derivation of the kinetic equation]\label{theor1}
We make the assumptions \eqref{as:pinitmoment}--\eqref{as:barpinit} on the initial data, and \eqref{eq:M}--\eqref{eq:lambda} on the coefficients. Let $p_\epsilon$ be the solution to~\eqref{eq:kineticscale}.
Then, for all $T>0$,  $\bar{p}_\epsilon$ is bounded in $L^\infty \big([0,T]\times \RR^d\times V\big)$ and
$$
\bar{p}_\epsilon :=\int_{\mathbb{R}}p_\epsilon\ud m    \underset{ \e \to 0 }{\rightharpoonup} \bar{p}_0 \quad \mbox{in } L^\infty \big([0,T]\times \RR^d\times V\big)\mbox{-weak-}\star
$$
and $\bar p_0$ satisfies equation \eqref{eq:kinetic}--\eqref{eq:tumbling1} with
$$
T\big(u, \bd{v},\bd{v}' \big)=\Lambda\Big(-\f{u}{G(0)},\bd{v},\bd{v}' \Big).
$$
Furthermore, we have $\bar{p}_0:=\int_{\mathbb{R}}p_0 \ud m$ with $p_0$ the weak limit (in measures, see~\eqref{est_L1}) of $p_\e$ which is given by
$$
p_0(\bd{x}, \bd{v},m,t) = \bar{p}_0 (\bd{x}, \bd{v},t) \delta \big(m=M(\bd{x},t) \big).
$$
\label{theo}
\end{theorem}
Before we prove this theorem in the next sections, we present a variant of this result.

\paragraph{Internal noise.}
Due to random receptor-methylation and demethylation events, some internal noise can be observed in {\em E. coli} chemotaxis and we can model it by adding a
diffusion term in $m$ \cite{E}. The model is as follows:
\beq \label{eq:kineticscalef}
{\p_t}p_\epsilon+\bd{v}\cdot\nabla_{\bd{x}} p_\epsilon+\f{1}{\epsilon}\p_m\Big(f\big(m-M\big)p_\epsilon\Big)=\epsilon \p^2_{mm} p_\epsilon  + {\mathcal Q_\e}[m,M](p_\e),
\eeq
with the no-flux boundary condition that now reads
\beq \label{eq:kineticscalefbc}
f(-M)  p_\e(\bd{x},\bd{v},m, t) - \e^2  \p_m p_\e(\bd{x},\bd{v},m, t) =0, \quad  \hbox{at } m =0.
\eeq
\begin{theorem}[Limit with noise] With the assumptions and notations of Theorem~\ref{theo},  the same conclusions hold for the solution $p_\e$ of~\eqref{eq:kineticscalef}, with the same expression for $p_0$ and
$$
T\big(u, \bd{v},\bd{v}' \big) =\sqrt{\f{G(0)}{2\pi}}\int_{\mathbb{R}}\Lambda\big(y,  \bd{v},\bd{v}' \big) e^{-\f{G(0)}{2}\big(y+\f{u}{G(0)}\big)^2}\ud y.
$$
\label{theoFP}
\end{theorem}

\paragraph{A priori bounds and principle of the proof.} Before we explain the derivation of the formula stated in these theorems, let us make some observations which explain the difficulty. Because we assume that $M(x,t)$ is given, we handle a linear equation for which existence and uniqueness of weak solutions is well established. The nonlinear case, when the chemoattractant concentration giving rise to $M$ is coupled to $p^\e$, can also be treated, see \cite{Liao_Jie}. In particular we will make use of the uniform estimates (see Section~\ref{sec:apb})
\beq \label{est_L1}
\dis \iiint_{\RR^d\times V\times \RR}  p_\e(\bd{x},\bd{v},m, t) \ud \bd{x} \ud \bd{v} \ud m = \dis \iiint_{\RR^d\times V\times \RR}  p^{\rm ini}(\bd{x},\bd{v},m) \ud \bd{x} \ud \bd{v} \ud m, \qquad \forall t \geq 0,
\eeq
\beq \label{est_L2}
\bar p_\e (\bd{x},\bd{v},t) \leq  \| \bar p^{\rm ini} (\bd{x},\bd{v}) \|_\infty e^{Ct}, \qquad \forall t \geq 0,
\eeq
where $C$ is a nonnegative constant.
From these bounds, we conclude that we can extract subsequences (but to simplify the notations we ignore this subsequence) which converge as mentioned in the theorems.

Passing to the limit in the equation on $p_\e$ (with or without noise) gives us
$$
\p_m\Big(f\big(m-M\big)p_0\Big)= 0.
$$
This tells us that $f\big(m-M\big)p_0 =0$ (it is constant and $p_0$ is integrable). Because, with assumption~\eqref{eq:F},  $f\big(m-M\big)$ vanishes only for $m-M=0$, we conclude that $p_0$ is a Dirac mass at $m=M$, hence the expression of $p_0$ in Theorems~\ref{theo} and \ref{theoFP}.

However this information is not enough to pass to limit in the equation on $\bar p_\e$ obtained integrating in $m$ equation~\eqref{eq:kineticscale} or \eqref{eq:kineticscalef}, that is
$$
{\p_t} \bar p_\epsilon+\bd{v}\cdot\nabla_{\bd{x}} \bar p_\epsilon = \int_{\RR^+}{\mathcal Q_\e}[m,M](p_\e) \ud m.
$$
Indeed, in the right hand side, the product $\Lambda \Big(\f{m-M}{\epsilon}, \bd{v},\bd{v}' \Big)p_\epsilon(\bd{x},\bd{v}',m,t) $ is, in the limit,  a discontinuity multiplied by a Dirac mass. For this reason, we have to rescale in $m$ in order to evaluate this limit, which we do in the next section.

\section{The change of variable}
\label{sec:change}

To get a more accurate view of the convergence of $p_\e$ to a Dirac mass in $m$, and following \cite{ErbanOthmer04}, we introduce a blow-up variable around $m=M$. We set
\beq
y=\f{m-M}{\epsilon},\qquad q_\epsilon(\bd{x},\bd{v},y,t)=\epsilon p_\epsilon(\bd{x},\bd{v},m,t)
\label{eq:yq}\eeq
so that
\beq
\bar q_\epsilon(\bd{x},\bd{v},t) :=  \int_{\mathbb{R}}q_\epsilon(\bd{x},\bd{v},y,t)\ud y =\int_{\mathbb{R}}p_\epsilon(\bd{x},\bd{v},m,t)\ud m= \bar p_\e (\bd{x},\bd{v},t).
\label{eq:barqe}\eeq
Because of these identities, our statements will equivalently be on $\bar q_\e$ and will go through the analysis of $q_\e$ rather than $p_\e$ itself.

Also notice that the bounds in \eqref{est_L1}, \eqref{est_L2} also hold true for $q_\e$ and $\bar q_\e$ and allow us to take weak limits.
\\

\noindent \textbf{(i) Without noise.}
The equation for $q_\epsilon(t,\bd{x},\bd{v},y)$ is written, using the definition in~\eqref{eq:DM},
$$ \bea
{\p_t}q_\epsilon+\bd{v}\cdot\nabla_{\bd{x}} q_\epsilon & - \dis\f{1}{\epsilon}D_t M\p_y q_\epsilon+\f{1}{\epsilon^2}\p_y\big(f(\epsilon y)q_\epsilon\big)
\\[15pt]
&= \dis \int_V \left[
\Lambda \Big(y, \bd{v},\bd{v}' \Big)q_\epsilon(\bd{x},\bd{v}',y,t) - \Lambda \Big(y, \bd{v}',\bd{v}\Big)q_\epsilon(\bd{x},\bd{v},y,t) \right] \ud\bd{v}' .
\eea
$$
From \eqref{eq:F}, we can write $f(\epsilon y)=\epsilon yG(\epsilon y)$ and the above equation becomes
\beq\label{eq:kineticq}
 \bea
{\p_t}q_\epsilon+\bd{v}\cdot\nabla_{\bd{x}} q_\epsilon & - \dis\f{1}{\epsilon}D_tM \p_y q_\epsilon-\f{1}{\epsilon}\p_y\big(yG(\epsilon y)q_\epsilon\big)
\\[15pt]
&= \dis \int_V \left[
\Lambda \Big(y, \bd{v},\bd{v}' \Big)q_\epsilon(\bd{x},\bd{v}',y,t) - \Lambda \Big(y, \bd{v}',\bd{v}\Big)q_\epsilon(\bd{x},\bd{v},y,t) \right] \ud\bd{v}' .
\eea
\eeq
Because $q_\e$ is a bounded measure on $\RR^d\times V \times \RR^+\times (0,T)$, for all $T>0$,
as $\epsilon\to 0$, $q_\e$ has a weak limit $q_0$ in the sense of measure
(again after extraction) and the above equation gives, in the distributional sense,
\beq\label{eq:kineticlim}
\p_y\big(yG(0)q_0+D_tM(S) q_0\big)=0.
\eeq
From this, we infer that
\beq\label{q0delta}
q_0(t,\bd{x},\bd{v},y)=\bar{q}_0(t,\bd{x},\bd{v})\delta\Big(y=-\f{D_tM(S)}{G(0)}\Big).
\eeq
This information is useful provided we can establish that
\beq\label{as:well}
\bar{q}_0(\bd{x},\bd{v},t)=\int_{\mathbb{R}}q_0(\bd{x},\bd{v},y,t)\ud y =  \mbox{weak-}\lim_{\e \to 0} \bar q_\e(\bd{x},\bd{v},t).
\eeq
This step is postponed to Section~\ref{sec:apb} and involves a control of the tail for large values of $m$.

We may also integrate equation~\eqref{eq:kineticq}  with respect to $y$ and find in the limit
\beq\label{eq:kineticqbar}
{\p_t}\bar{q}_0+\bd{v}\cdot\nabla_{\bd{x}}\bar{q}_0
= \dis \int_V \left[
\Lambda \Big(-\f{D'_tM}{G(0)}, \bd{v},\bd{v}' \Big)\bar{q}_0'\ud\bd{v}'-\Lambda \Big(-\f{D_tM}{G(0)}, \bd{v}',\bd{v} \Big) \bar{q}_0\right],
\eeq
where $D'_tM(S)$ is the total derivative, as in~\eqref{eq:DM}, but in the direction $\bd{v}'$ and where $\bar{q}_0'$ represents $\bar{q}_0(\bd{x},\bd{v}',t)$.
Finally, for any smooth  test function $\phi$, we have from the change of variable $m\mapsto y=(m-M)/\e$,
$$
\int_\RR p_\e(\bd{x},\bd{v},m,t)\phi(m)\ud m = \int_\RR q_\e(\bd{x},\bd{v},y,t) \phi(M+\e y)\ud y
\to \bar{q}_0(\bd{x},\bd{v},t)\phi(M),
$$
where we use \eqref{q0delta}. This gives the limiting expression of $p_0$ in Theorem \ref{theo}.

These are the results stated in Theorem~\ref{theo}, if we can establish the relation $\bar q_0=\bar p_0$ as stated in~\eqref{as:well}, which we do later.
\\

\noindent \textbf{(ii) With internal noise.}
Similarly, after introducing the new variables as in \eqref{eq:yq},
the equation~\eqref{eq:kineticscalef} for $q_\epsilon(t,\bd{x},\bd{v},y)$ writes
$$ \bea
{\p_t}q_\epsilon+\bd{v}\cdot\nabla_{\bd{x}} q_\epsilon& -\dis \f{1}{\epsilon}D_tM \p_y q_\epsilon+\f{1}{\epsilon^2}\p_y\big(f(\epsilon y)q_\epsilon\big)
\\[15pt]
& =\f{1}{\epsilon}\p_{yy}^2 q_\epsilon + \dis \int_V \left[
\Lambda \Big(y, \bd{v},\bd{v}' \Big)q_\epsilon(\bd{x},\bd{v}',y,t) - \Lambda \Big(y, \bd{v}',\bd{v}\Big)q_\epsilon(\bd{x},\bd{v},y,t) \right] \ud\bd{v}' .
\eea
$$
From \eqref{eq:F}, this equation becomes
\beq\label{eq:kineticqfm}
\bea
{\p_t}q_\epsilon+\bd{v}\cdot\nabla_{\bd{x}} q_\epsilon&- \dis\f{1}{\epsilon}D_tM\p_y q_\epsilon-\f{1}{\epsilon}\p_y\big(yG(\epsilon y)q_\epsilon\big)
\\[15pt]
& =\f{1}{\epsilon}\p_{yy}^2 q_\epsilon+ \dis \int_V \left[
\Lambda \Big(y, \bd{v},\bd{v}' \Big)q_\epsilon(\bd{x},\bd{v}',y,t) - \Lambda \Big(y, \bd{v}',\bd{v}\Big)q_\epsilon(\bd{x},\bd{v},y,t) \right] \ud\bd{v}' .
\eea
\eeq
In the limit $\epsilon\to 0$, the above equation converges to, in the sense of distributions,
$$
\p_y\big(yG(0)q_0+D_tM q_0\big)=-\p_{yy}^2q_0 ,
$$
which shows that
\beq\label{q0gaussian}
q_0(\bd{x},\bd{v},y,t)=\bar{q}_0 (\bd{x},\bd{v},t) \; \sqrt{\f{G(0)}{2\pi}} e^{ -\f{G(0)}{2}\big(y+\f{D_tM}{G(0)}\big)^2},
\eeq
a useful information, still assuming we have proved the relation~\eqref{as:well} for $\bar q_0$.

We conclude as before. After integration of \eqref{eq:kineticqfm} with respect to $y$, passing to the limit $\e \to 0$, we find
\beq\label{eq:kineticqbarfm} \begin{cases}
{\p_t}\bar{q}_0+\bd{v}\cdot\nabla_{\bd{x}}\bar{q}_0
= \dis \int_V \left[T\Big(D'_tM,\bd{v}',\bd{v} \Big)\bar{q}_0'\ - T\Big(D_tM, \bd{v}, \bd{v}'\Big)\bar{q}_0 \right] ud\bd{v}',
\\[15pt]
T\Big(D_tM, \bd{v}, \bd{v}'\Big)=\sqrt{\f{G(0)}{2\pi}} \dis \int_{\mathbb{R}}\Lambda(y , \bd{v}, \bd{v}')e^{-\f{G(0)}{2}\big(y+\f{D_tM}{G(0)}\big)^2}
\ud y.
\end{cases}
\eeq
The Theorem~\ref{theoFP} is also proved.\qed

\section{A priori bounds}
\label{sec:apb}

We now establish  the various estimates which justify that we can pass to the limit as indicated in Section~\ref{sec:change} and thus we prove the
\begin{lemma} We make the  assumptions of Theorem \ref{theo}, then the condition~\eqref{as:well} holds and for some constant which depends on $\iint y^2q^{\rm ini}\ud y\ud x\ud v $ and $\| M \|_{W^{1,\infty}(\RR^d\times \RR^+)}$, we have
$$
\iint y^2q_\epsilon(t)\ud y\ud x\ud v \leq C( q^{\rm ini},M).
$$
Consequently,  $q_\epsilon$ converges weakly in the sense of measure towards $q_0$ and
\\
(i) for  $q_\epsilon$ a solution to \eqref{eq:kineticq}, $q_0$ is given by \eqref{q0delta} with $\bar{q}_0$ weak solution of \eqref{eq:kineticqbar},
\\
(ii)  for $q_\epsilon$ a solution to \eqref{eq:kineticqfm}, then, $q_0$ is given by \eqref{q0gaussian} with $\bar{q}_0$ weak solution of \eqref{eq:kineticqbarfm}.
\end{lemma}

\begin{proof}
We only consider the case (i) without noise, the case (ii) is obtained by the same token. We first prove some estimates which imply weak convergence. Then, we pass to the limit in the equation satisfied by $\bar{q}_\epsilon$.
\\

\noindent {\bf $L^1$ bound.} For completness, we recall that equation \eqref{eq:kineticq} is positivity preserving and conservative. It follows the uniform, in $\epsilon$, bound for $q_\epsilon$ in $L^1$, see~\eqref{est_L1}.
\\

\noindent{\bf $L^\infty$ bound on  $\bar{q}_\epsilon$.}
We use the notation~\eqref{eq:barqe} for $\bar{q}_\epsilon$. Arguing in the spirit of \cite{HKS,bcgp,krm}), we first prove the uniform $L^\infty$ bound on $\bar{q}_\epsilon$.

Integrating~\eqref{eq:kineticq} with respect to $y$, from the bound~\eqref{eq:lambda}
and the nonnegativity of $q_\epsilon$, we get
$$
\p_t\bar{q}_\epsilon+\bd{v}\cdot\nabla_{\bd{x}}\bar{q}_\epsilon \leq \lb_+ \int\bar{q}_\epsilon\ud\bd{v}'.
$$
Then, using the method of characteristics, we have
$\p_t\bar{q}_\epsilon(t,\bd{x}+\bd{v}t)\leq \lb_+ \int\bar{q}_\epsilon (\bd{x} + \bd{v}t, \bd{v}',t) \ud\bd{v}'$,
which implies after integration
$$
\bar{q}_\epsilon(t,\bd{x},\bd{v})-\bar{q}^{\rm ini}(\bd{x}-\bd{v}t)\leq \lb_+  \int_0^t\int\bar{q}_\epsilon (\bd{x} - \bd{v}s, \bd{v}',t-s) \ud\bd{v}' \ud s.
$$
Taking the supremum in $\bd{x}$, $\bd{v}$, we find
$$
\| \bar q_\e(t) \|_\infty \leq \| \bar q^{\rm ini} \|_\infty +  \lb_+ \, |V| \, \int_0^t \| \bar q_\e(s) \|_\infty ds,
$$
and using Gronwall's inequality, we find the estimate in~\eqref{est_L2}.
\\

\noindent{\bf Control on the tail in $m$.} In order to prove the condition~\eqref{as:well}, we need to ensure that there is no mass loss at
infinity in $m$. To do so,  we multiply both sides of \eqref{eq:kineticq} by $y^2$ and integrate by parts
with respect to $x$, $v$, and $y$. This yields
$$
\frac{d}{dt} \iint y^2q_\epsilon\ud y\ud x\ud v + \frac{2}{\epsilon} \iint y^2G(\epsilon y)q_\epsilon\ud y\ud x\ud v + \frac{2}{\epsilon} \iint yD_tM q_\epsilon \ud y\ud x\ud v=0 .
$$
Using the Cauchy-Schwarz inequality, we deduce
$$
\begin{aligned}
\frac{d}{dt} \iint y^2q_\epsilon\ud y\ud x\ud v + \frac{2}{\epsilon} \iint y^2G(\epsilon y)q_\epsilon\ud y\ud x\ud v \leq
& \f{1}{\epsilon} \iint y^2G(\epsilon y)q_\epsilon \ud y \ud x\ud v \\
& + \f{1}{\epsilon} \iint (D_tM)^2 \frac{q_\epsilon}{G(\epsilon y)}\ud y \ud x\ud v.
\end{aligned}
$$
By assumption~\eqref{eq:M},  $D_tM$ is bounded in $L^{\infty}([0,T]\times \RR^d\times V)$.
From assumption~\eqref{eq:F} and the mass conservation, the last integral of the right hand side is uniformly bounded by a constant denoted by $C>0$.
Then from assumption \eqref{eq:F}, we have,
$$
\frac{d}{dt} \iint y^2q_\epsilon\ud y\ud x\ud v + \frac{g_-}{\epsilon} \iint y^2 q_\epsilon\ud y\ud x\ud v \leq \frac{C}{\epsilon}.
$$
From the Gronwall Lemma, we deduce the bound for all $t>0$,
$$
\iint y^2q_\epsilon(t)\ud y\ud x\ud v \leq e^{-t g_-/\epsilon}\iint y^2q^{\rm ini}\ud y\ud x\ud v +\frac{C}{g_-},
$$
which implies a uniform bound on $\iint y^2q_\epsilon(t)\ud y\ud x\ud v$.
\\

\noindent{\bf Passing to the limit.}
From the bound above, we deduce that, we can extract a subsequence which converges weakly
in measure $q_\epsilon\rightharpoonup q_0$ and such that
$\bar{q}_\epsilon\rightharpoonup \bar{q}_0$ in $L^\infty$-weak$\star$.
Then we can pass to the limit in the sense of distribution in equation \eqref{eq:kineticq}
and deduce that the limit $q_0$ satisfies equation \eqref{eq:kineticlim} in the sense of distribution.
In fact, we notice that from the Lipschitz character of $G$, we have
$$
\iint y(G(\epsilon y)-G(0))\bar{q}_\epsilon \ud x \ud y \leq C \epsilon \iint y^2 \bar{q}_\epsilon \ud x \ud y \to 0, \qquad \mbox{ as } \epsilon \to 0,
$$
thanks to the estimate on the tail above.
Finally, \eqref{eq:kineticlim} implies that $yG(0) q_0 + D_t M q_0$ is constant a.e.,
and this constant should be $0$ since from the estimate on the tail above, we have that
$yq_0 \in L^1$. We conclude that $q_0$ vanishes except when $y G(0)+D_tM(S)=0$.
By conservation of the mass, we deduce the expression~\eqref{q0delta} for $p_0$.
\end{proof}

\section{Comments on physical background}
\label{sec:physics}

 The form of the equation  \eqref{eq:kineticscale} corresponds, for {\em E. coli} chemotaxis, to the formalism in the physical literature. We have simplified the notations for mathematical clarity and we explain now how to relate our notations to known biophysical quantities. Here we have used the same biological parameters
  as in \cite{JOT, SWOT}.

\begin{itemize}
\item The methylation level $M(x,t)$ at equilibrium is related to the extra-cellular attractant profile $S$, by a logarithmic dependency
$$M=M(S)=m_0+\f{f_0(S)}{\alpha},\qquad\mbox{with}\quad
f_0(S)=\ln\biggl(\f{1+S/K_I}{1+S/K_A}\biggr).$$
The constants $m_0$, $K_I$, $K_A$
represent the basic methylation
level, and the dissociation constants for inactive, respectively active, receptors.
Numerical values are given by $\alpha=1.7$, $m_0=1$, $K_I=18.2\mu M$, $K_A=3mM$.

\item The receptor activity $a(m,S)$ depends on the intracellular
methylation level $m$ and the extracellular chemoattractant
concentration $S$ such that
\beq\label{eq:am}
a=\bigl(1+\exp(NE)\bigr)^{-1},\qquad\mbox{with }
E=-\alpha(m-m_0)+f_0(S)= - \alpha(m-M(S)) .
\eeq
The coefficient $N=6$ represents the number of tightly coupled receptors.

\item The intracellular dynamics and tumbling frequency are given by
$$
f\big(m-M(S)\big) = F(a)=k_R(1-a/a_0),\qquad \lambda \big(m-M(S) \big)= Z(a)=z_0+\tau^{-1}\big(\f{a}{a_0}\big)^H,
$$
where $a(m,S)$ is the receptor activity defined in \eqref{eq:am}.
The parameter $k_R$ is the
methylation rate, $a_0$ is the receptor preferred activity which is such that
$f(a_0)=0$ and $f'(a_0)<0$. For the tumbling frequency,
$z_0$, $H$, $\tau$ represent the rotational diffusion, the Hill coefficient
of flagellar motors response curve and the average run time
respectively. All these parameters can be measured biologically, their values are given by
$k_R=0.01s^{-1}\sim 0.0005s^{-1} $, $a_0=0.5$, $z_0=0.14s^{-1}$, $\tau=0.8s$, $H=10$.

\item Two kinds of noise can be observed in the signaling pathway for {\it E. coli}, one is from the external fluctuation of
the ligand concentration and the other is the internal noise from the random receptor-methylation and demethylation event \cite{E}.
For small complexes, the effect on the activity from the external noise is negligible compared to the internal noise.
\end{itemize}
We refer the readers to \cite{SWOT,E,JOT, STY} and the
references therein for the detailed physical meanings of these
parameters.

In {\em E. coli} chemotaxis, the adaptation time is faster than the system time, i.e. the rescaled constant $k_R$ is large.
Let $k_R=\f{1}{\epsilon}$, where $\epsilon$ is the ratio between the system time and the adaptation time. For example, when the system time is 1000s and
the adaptation time is 100s, $\epsilon=0.1$.

Moreover, the
Hill coefficient $H$ is large which indicates that $\lambda(m-M(S))$ varies fast with respect to $m-M(S)$.
Therefore, the scaling introduced in \eqref{eq:kineticscale} is satisfied by {\it E. coli} chemotaxis.
We can use
$$
 \quad f(r)=1-\f{1}{a_0}\big(1+\exp(-N\alpha r)\big)^{-1}.
$$
Therefore, the function $G(\cdot)$ used in~\eqref{eq:F} is given by
$$
G(r)=-\f{f(r)}{r}=-\f{1}{r}\Big(1-\f{1}{a_0}\big(1+\exp(-N\alpha r)\big)^{-1}\Big),\qquad
\mbox{with }G(0)=-f'(0)=\f{N\alpha}{4a_0}.
$$
Besides, from Theorem \ref{theor1} for the case without noise, using $y=\f{r}{\epsilon}$ yields
$$\begin{array}{rl}
T(D_tM,\bd{v},\bd{v'}) & =\Lambda(y,\bd{v},\bd{v}')=z_0+\tau^{-1} \dis \f{\big(1+\exp(-N\alpha \epsilon y)\big)^{-H}}{a_0^H}
\\[5pt]
&=
z_0+\tau^{-1} \dis \f{\Big(1+\exp\big(N\alpha \epsilon D_tM(S)/G(0)\big)\Big)^{-H}}{a_0^H}.
\end{array}
$$
And from Theorem \ref{theor1}, when the noise in the methylation level is considered,
$$
T\big(D_tM, \bd{v},\bd{v}' \big) =
\sqrt{\f{G(0)}{2\pi}}\int_{\mathbb{R}}\Lambda\big(y,  \bd{v},\bd{v}' \big) e^{-\f{G(0)}{2}\big(y+\f{D_tM}{G(0)}\big)^2}\ud y.
$$

Since the run duration last longer when bacteria encounter an increasing gradient of chemotactic molecules, this lead
to higher bacteria density at the place where the ligand concentration $S$ is higher. This phenomena is well explained by
the classical Keller-Segel model which can be considered as the parabolic limit of the kinetic-transport models.
However, recent experimental observation shows that higher ligand concentration leading to higher bacteria density is only valid
 in a spatial-temporal slow-varying environment. When the ligand concentration varies fast,
 there exists a phase shift between the mass center of ligand concentration and of the
 bacterias \cite{ZSOT}. This is due to the memory effect in the slow methylation adaptation rate.
In the limiting kinetic model, $D_tM(S)$ takes into account the memory effect along the trajectory of the moving cell \cite{krm}.
Then, two interesting questions come: are these two kinds of memory effect the same?
Can we see the phase shift between the mass center of ligand concentration
and of the bacteria in the limiting kinetic model?

\section{Numerical illustrations}
\label{sec:num}

We performed numerical simulations using the method SPECS  \cite{JOT}. It  is a cell based model that takes into account the evolution of each cell intracellular methylation level, which
determines the tumbling frequency of each bacteria. As explained in \cite{STY,SWOT}, SPECS
and the kinetic model that incorporates intracellular chemo-sensory system \eqref{eq:kineticm} show a quantitative match.
As in \cite{STY,SWOT}, we choose the two velocity kinetic model and
a periodic 1D traveling wave concentration
$$
S(x,t) = S_0 +S_A
\sin[\f{2\pi}{\ell}(x-ut)] ,
$$
which is spatial-temporal varying
and where the wavelength $\ell$ is equal to the length of the domain.
We compare the numerical results of SPECS and the limiting kinetic model in Figure~\ref{fig:numerical}. Upwind scheme
is used for the transport terms and periodic boundary conditions are considered.
The density is scaled at the order of $10^{-3}$, it is the ratio between the actual cell number and the total number. In the
Figure~\ref{fig:numerical}, the $x$ axis represents the remainder of $x-ut$ mod $\ell$, i.e. we keep tracking the wave front of the periodic traveling wave.

 Two different wave velocities $u = 0.4\mu m/s$ and $u=8\mu m/s$ are considered.
We compare the density profiles $\rho=\int_V p \ud\bd{v}$ and the cell flux $J=\int_V\bd{v}p\ud\bd{v}$.
 When the concentrated wave moves
slowly, the limiting kinetic model gives good consistency, however
in the fast-varying environment, the density  and cell flux profiles are different for
SPECS and the limiting kinetic model. We refer
the readers to \cite{JOT,SWOT} for more detailed
discussions and physical explanations.

This numerical experiment shows that the memory effect using the model based on $D_tM(S)$ is different from the memory effect for the complete model~\eqref{eq:kineticm} when  fast external chemoattractant waves are considered. This phenomena can be explained by the slow adaptation rate in the methylation level and the memory along the trajectory compared to the phase shift. In this fast wave regime, our mathematical results do not apply because the scaling assumptions are not satisfied.

\begin{figure}
\centering
\includegraphics[width=3.25in]{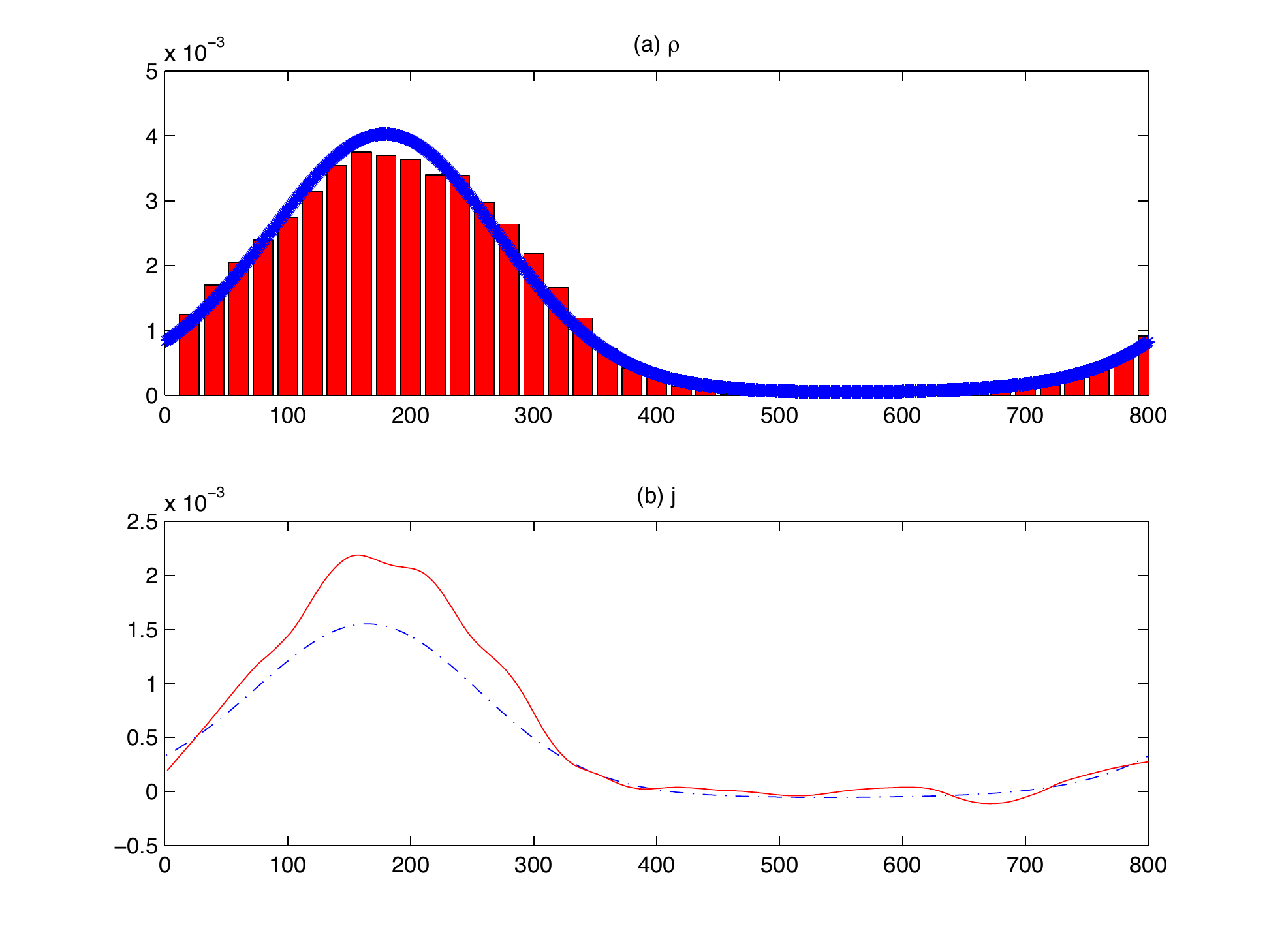}
\includegraphics[width=3.25in]{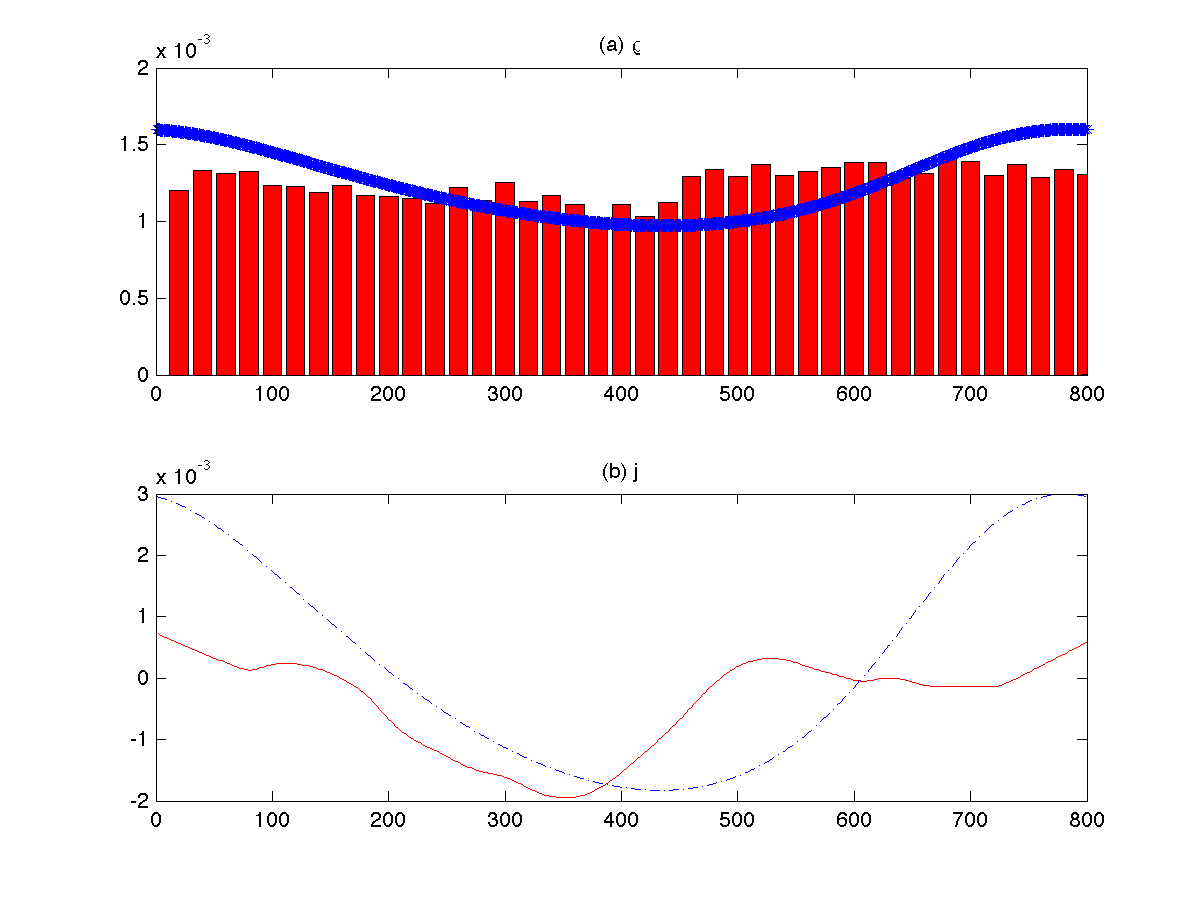}
\vspace{-6mm}
\caption{Numerical comparison between the limiting kinetic model
and SPECS. The steady state profiles of density $\rho$ (top) and current $J$ (bottom) are
presented. Left: $u=0.4\mu m/s$; Right: $u=8\mu m/s$. In the
subfigures, histograms and solid lines are from SPECS, while
the dash star and dash dotted lines are calculated using the limiting kinetic equation.
Parameters used here are $[L]_0
= 500\mu M$, $[L]_A = 100\mu M$, $\ell = 800 \mu m$.  We have used $20,000$
cells for simulation with SPECS. }\label{fig:numerical}
\end{figure}


\bibliography{swarm}

\end{document}